\documentclass[reqno,11pt]{amsart}
\usepackage{color}
\usepackage{mathtools}
\usepackage{amsfonts}
\usepackage{amssymb}
\usepackage{amsthm}
\usepackage{newlfont}
\usepackage{graphicx}
\usepackage{float}
\usepackage{amsmath}
\usepackage{multirow}

\usepackage{geometry}
\usepackage{tikz}
\usetikzlibrary{matrix,shapes,arrows,positioning,chains}

\numberwithin{equation}{section}
\newtheorem{thm}{Theorem}[section]
\newtheorem{lem}{Lemma}[section]
\newtheorem{rem}{Remark}[section]

\newtheorem{prop}{Proposition}[section]

\newcommand{\ed}{\end {document}}

\begin{document}

\title[Stability of Strang splitting method]{Stability and convergence of Strang splitting. Part II:  Tensorial Allen-Cahn equations}

\author[D. Li]{ Dong Li}
\address{D. Li, SUSTech International Center for Mathematics, and Department of Mathematics,  Southern University of Science and Technology,
	Shenzhen, P.R. China}
\email{lid@sustech.edu.cn}

\author[C.Y. Quan]{Chaoyu Quan}	
\address{C.Y. Quan, SUSTech International Center for Mathematics, and Department of Mathematics,  Southern University of Science and Technology,
	Shenzhen, P.R. China}
\email{quancy@sustech.edu.cn}

\author[J. Xu]{Jiao Xu}	
\address{J. Xu, SUSTech International Center for Mathematics,  Southern University of Science and Technology,
	Shenzhen, China }
\email{xuj7@sustech.edu.cn}

\maketitle

\begin{abstract}
We consider the second-order in time  Strang-splitting approximation for tensorial (e.g. vector-valued
and matrix-valued) Allen-Cahn equations.  Both the linear propagator and the nonlinear propagator
are computed explicitly. For the vector-valued case, we prove the maximum principle and unconditional
energy dissipation for a judiciously modified energy functional. The modified energy functional is close
to the classical energy up to $\mathcal O(\tau)$ where $\tau$ is the splitting step. For the
matrix-valued case, we prove a sharp maximum principle in the matrix Frobenius norm.
We show modified energy dissipation under very mild splitting step constraints. We exhibit 
several numerical examples to show the efficiency of the method
as well as  the sharpness of the results.
\end{abstract}

\section{Introduction}
In this work we investigate the stability of second-order in time  Strang-splitting methods applied
to two models: One is the vector-valued Allen-Cahn (AC) equation, and the other is the matrix-valued
Allen-Cahn system. The operator splitting methods have been extensively used
 in the numerical simulation of many physical problems,
including phase-field equations \cite{CKQT15, WT16, LQ1a, LQ1b, LQ1c, L21, LQ1d, K10}, 
Schr\"odinger equations \cite{BJM02, T12, LW20}, and  the reaction-diffusion systems \cite{D01, Nie11}.  A prototypical second order in time method is the Strang splitting approximation
\cite{St68, M90}. In a slightly more general set up,  we consider the following 
abstract parabolic problem:
\begin{align} \label{a0}
\begin{cases}
\partial_t u = \mathcal L u - f(u), \quad t>0; \\
u\Bigr|_{t=0} = u_0.
\end{cases}
\end{align}
where $u: [0, \infty) \to \mathbb B$ and $\mathbb B$ is a real Banach space.  The operator
$\mathcal L: \; \mathcal D(\mathcal L) \subset \mathbb B \to \mathbb B$ is a dissipative closed operator
which typically is the infinitesimal generator of a strongly-continuous dissipative semigroup. 
The domain $\mathcal D (\mathcal L)$ is typically a dense subset of $\mathbb B$. 
On the other hand $f:\, \mathbb B \to \mathbb B$ is a nonlinear operator.  
Take $\tau>0$ as the splitting time step. Define for $t>0$,
\begin{align}
\mathcal S_{\mathcal  L}^{(\mathbb B)} (t) = e^{t\mathcal L},
\end{align}
which is the solution operator to the linear equation $\partial_t u = \mathcal L u$. 
Define $\mathcal S_{\mathcal N}^{(\mathbb B)} (\tau)$ as the nonlinear solution operator to
the system
\begin{align}
\begin{cases}
\partial_t u = - f (u), \; 0<t \le \tau; \\
u\Bigr|_{t=0} =b \in \mathbb B.
\end{cases}
\end{align}
In yet other words, $\mathcal S_{\mathcal N}^{(\mathbb B)} (\tau)$
is the map $b\to u(\tau)$.   The Strang-splitting approximation for
\eqref{a0} takes the form
\begin{align} \label{a4}
u^{n+1} = \mathcal S^{(\mathbb B)}_{\mathcal L}(\tau/2 ) \mathcal S^{(\mathbb B)}_{\mathcal N}(\tau)
\mathcal S_{\mathcal L}^{(\mathbb B)}  (\tau/2)u^n, \quad n\ge 0.
\end{align}
Let $u^{\mathrm{ex}} $ be the exact solution to \eqref{a0}. In general it is expected that
on a finite time interval $[0, T]$
\begin{align} \label{a6}
\sup_{n\tau \le T} \| u^{\mathrm{ex} }(n\tau)  -u^{n}  \| = \mathcal O(\tau^2),
\end{align}
where $\| \cdot \|$ denotes a working norm which is allowed to be weaker than the norm endowed 
with the Banach space $\mathbb B$.  On the other hand, the local truncation error is
typically $\mathcal O(\tau^3)$, i.e.
\begin{align} \label{a8}
\| \mathcal S^{\mathrm{ex} } (\tau) a -  \mathcal S^{(\mathbb B)}_{\mathcal L}(\tau/2 ) \mathcal S^{(\mathbb B)}_{\mathcal N}(\tau)
\mathcal S_{\mathcal L}^{(\mathbb B)}  (\tau/2) a \| = \mathcal O(\tau^3),
\end{align}
where $a\in \mathbb B$, and $\mathcal S^{\mathrm{ex} } (\tau)$ is the exact solution operator
to \eqref{a0}. 

Despite the remarkable effectiveness of the scheme \eqref{a4}
(cf. \cite{WT16} for the case of Allen-Cahn equations),  there have been few rigorous
works addressing the stability and regularity of the Strang splitting solutions. 
The general assertions \eqref{a6}--\eqref{a8}  are hinged on nontrivial a priori estimates of
 the numerical iterates in various Banach spaces.  In a recent series of works
 \cite{LQ1a, LQ1b, LQ1c, LQ1d, L21}, we developed  a new theoretical framework
 to establish convergence,  stability and regularity of general operator splitting methods for a myriad
 of phase field models including the  Cahn-Hilliard equations, Allen-Cahn equations and the like.
 More pertinent to the discussion here is the recent work \cite{LQ1d} which settled the stability for a class of scalar-valued Allen-Cahn equations with  polynomial or  logarithmic potential nonlinearities.

In this work we develop further the program initiated in \cite{LQ1d} and analyze the Strang-splitting for two types of tensorial Allen-Cahn equations. 
The first model is the vector-valued Allen-Cahn equation
\begin{align}
\partial_t {\mathbf u} = \Delta \mathbf u + (1-|\mathbf u|^2) \mathbf u,
\quad \mathbf u\in \mathbb R^m, 
\end{align}
where $|\mathbf u|^2 = \sum_{i=1}^m u_i^2$ for $\mathbf u =(u_1,\cdots,u_m)^{\mathrm T}$
($m\ge 1$ is an integer),  and the Laplacian operator is applied to $\mathbf u$ component-wise. 
In particular, when $m=2$, the vector-valued Allen-Cahn equation is equivalent to the complex-valued Ginzburg–Landau model for superconductivity (cf. \cite{JT90,EH94} using a standard transformation $\psi(x) = e^{iA\cdot x} u(x)$ and identifying $u = u_1+i u_2$).
The second model is the matrix-valued Allen-Cahn equation:
\begin{align}
\partial_t U = \Delta U + U-UU^{\mathrm T}U,
\quad U\in \mathbb R^{m\times m},
\end{align}
where the Laplacian operator is applied to the matrix $U$ entry-wise. 
The matrix-valued Allen-Cahn is introduced in \cite{OW20} to find stationary points of the Dirichlet energy for orthogonal matrix-valued fields, that can be used in inverse problems in image analysis and  directional field synthesis problems etc.
For both models we shall develop the corresponding stability theory for the Strang-splitting approximation in the style of \eqref{a4}. 
Roughly speaking, our results can be summarized in the following table where $\tau$ is the splitting time step.
\begin{table}[!h]
\centering
\renewcommand\arraystretch{1.75}
\begin{tabular}{|l|l|l|}
\hline
\cline{1-3}
 $\;$ & $L^{\infty}$-stability  &   Modified energy dissipation\\
\hline
 Vector-valued AC  & $0<\tau<\infty$  &  $0<\tau<\infty$ \\
\hline
Matrix-valued AC &  $0<\tau<\infty$ & $m e^\tau (e^{2\tau}-1)\le 0.43$ \\
\hline
\cline{1-3}
\end{tabular}
\end{table}

We turn now to more precise formulation of the results. Our first result is on the vector-valued
Allen-Cahn equation. An interesting feature is that the nonlinear propagator still enjoys
a relatively simple explicit expression.  

\begin{thm}[Unconditional stability of Strang-splitting for the vector-valued AC] \label{thm1}
Suppose  $\Omega =[-\pi, \pi]^d$ is the $2\pi$-periodic 
$d$-dimensional torus in physical dimensions $d\le 3$. 
 Consider the vector-valued AC for $\mathbf u:\, [0,\infty)
\times \Omega \to \mathbb R^m$, $m\ge 1$:
\begin{align}
\begin{cases}
\partial_t \mathbf u =  \Delta \mathbf u + (1-|\mathbf u|^2) \mathbf u, 
\quad (t,x) \in (0,\infty) \times \Omega; \\
\mathbf u\Bigr|_{t=0} = \mathbf u^0,
\end{cases}
\end{align}
where $\mathbf u^0:\, \Omega \to \mathbb R^m$ is the initial data.
Define $\mathcal S_{\mathcal L}(t) = e^{t \Delta}$ and  for $\mathbf w\in
\mathbb R^{m}$, 
\begin{align}
\mathcal S_{\mathcal N}(t) \mathbf w \coloneqq 
\left((e^{2t}-1) |\mathbf w|^2+1\right)^{-\frac12} e^t
\mathbf w.
\end{align}
Define for $n\ge 0$ the Strang-splitting iterates
\begin{align}
\mathbf u^{n+1}  = \mathcal S_{\mathcal L}\left( \tau/2\right) \mathcal S_{\mathcal N}\left( \tau \right)\mathcal S_{\mathcal L}\left( \tau/2\right) \mathbf u^n.
\end{align}
The following hold.

\begin{enumerate}
\item \underline{The maximum principle}. For any $\tau>0$ and any $n\ge 0$, it holds that
\begin{align}
\|  |\mathbf u^{n+1}|  \|_{L_x^\infty} \le \max \{ 1, \; \| |\mathbf u^n| \|_{L_x^\infty} \}.
\end{align}
It follows that
\begin{align}
\sup_{n\ge 1} \| |\mathbf u^n| \|_{L_x^\infty} \le \max\{1, \; \| |\mathbf u^0| \|_{L_x^\infty} \}.
\end{align}
In  particular if $\| |\mathbf u^0| \|_{L_x^\infty} \le 1$, then
\begin{align}
\sup_{n\ge 1} \| |\mathbf u^n| \|_{L_x^\infty} \le 1.
\end{align}

\item \underline{Modified energy dissipation}. For any $\tau>0$, we have
\begin{align}
\widetilde E(\tilde {\mathbf u}^{n+1})\le \widetilde E(\tilde {\mathbf u}^{n}),\quad \forall\, n\ge 0.
\end{align}
Here 
\begin{align}
& \tilde{\mathbf u}^n = \mathcal S_{\mathcal L}\left( \tau/2\right)\mathbf u^n ; \\
&\widetilde E({\mathbf u}) =\int_\Omega \left( \frac{1}{2\tau}\left\langle (e^{-\tau \Delta} -1)  {\mathbf u},  {\mathbf u} \right\rangle + G({\mathbf u})  \right)\, dx;\\
& G(\mathbf u)  = \frac1{2\tau} |\mathbf u|^2 - \frac{e^\tau}{\tau(e^{2\tau}-1)} \left(\left(
1+(e^{2\tau}-1)|\mathbf u|^2\right)^{\frac12}-1\right).
\end{align}
In the above $\langle \mathbf a, \, \mathbf b \rangle = \sum_{i=1}^m a_i b_i$
for $\mathbf a = (a_1,\cdots, a_m)^{\mathrm T}$,
$\mathbf b =(b_1,\cdots, b_m)^{\mathrm T} \in \mathbb R^m$.
\end{enumerate}
\end{thm}

\begin{rem}
The significance of the uniform stability result obtained in Theorem \ref{thm1} is that it leads to 
all higher Sobolev estimates as well as convergence.  For example, by using the techniques
developed in \cite{LQ1d}, we can show uniform Sobolev bounds. Namely if  $\mathbf 
u^0 \in H^{k_0}(\Omega, \, \mathbb R^m)$ for some
$k_0\ge 1$, then
\begin{align}
\sup_{n\ge 1} \| \mathbf u^n \|_{H^{k_0}} \le C_1,
\end{align}
where $C_1>0$ depends only on ($k_0$, $d$, $\|u^0 \|_{H^{k_0}}$, $m$). 
Moreover for any $k\ge k_0$, we have
\begin{align}
\sup_{n\ge \frac {1} {\tau} } \| \mathbf u^n \|_{H^{k}}
\le C_2,
\end{align}
where $C_2>0$ depends only on ($m$, $k$, $k_0$, $d$, $\| u^0 \|_{H^{k_0} }$).
Let $\mathbf u^{\mathrm {ex}}$ be the exact solution to the vector-valued Allen-Cahn
equation corresponding to initial data $\mathbf u^0$.
If we assume $\mathbf u^0$ has high regularity (e.g. $\mathbf u^0\in H^{k_0}$
for some sufficiently large $k_0$), then for any $T>0$, it holds that
\begin{align} 
\sup_{n\ge 1, n\tau \le T}  \| \mathbf u^n - \mathbf u^{\mathrm {ex}} (n\tau, \cdot ) \|_{L^2(\Omega)}
\le C \cdot \tau^2,
\end{align}
where $C>0$ depends on ($\mathbf u^0$, $T$, $d$, $m$).
\end{rem}
\begin{rem}
The modified energy for $\mathbf u^n$ has a close connection with the standard energy 
$E^{\mathrm{st}}(\mathbf u^n)$ defined by
\begin{align}
E^{\mathrm {st}} (\mathbf u^n) = \int_{\Omega}
\left( \frac 12 | \nabla \mathbf u^n|^2 +
\frac 14 ( | \mathbf u^n|^2 -1)^2 -\frac 14 \right) \,dx.
\end{align}
Note that here  the integrand of $E^{\mathrm{st}}(\cdot)$ includes a harmless constant $-\frac 14$. 
If $\mathbf u^0 \in H^{k_0}(\Omega, \mathbb R^m)$ for sufficiently
large $k_0$, then  for $0<\tau\le 1$, we have
\begin{align} 
\sup_{n\ge 0} | \widetilde E^n - E^{\mathrm{st}}(\mathbf u^n ) | \le C_3 \tau,
\end{align}
where $C_3>0$ depends only on ($d$, $m$, $\mathbf u^0$). This result can be proved
by using the uniform Sobolev estimates established in the preceding remark. We omit
the elementary argument here for simplicity.
\end{rem}

In recent work \cite{OW20}, Osting and Wang considered the minimization problem 
\begin{align}
 \min_{ A\in H^1(\Omega, O_m)} \frac 12 \int_{\Omega} \| \nabla A \|_F^2 \,dx,
 \end{align}
 where $O_m\subset \mathbb R^{m\times m}$ is the group of orthogonal matrices, and the gradient
 is taken as the usual sense when $A$ is regarded as a matrix-valued function in $\mathbb R^{m\times m}$, i.e. \emph{not} the covariant derivative sense in e.g. \cite{BB14}. 
For a matrix $A, B\in \mathbb R^{m\times m}$, we use the usual Frobenius norm and Frobenius
inner product:
\begin{align}
&\|A \|_F^2 = \sum_{i,j=1}^m A_{ij}^2,  \quad
\langle A, B \rangle_F = \sum_{i,j=1}^m A_{ij} B_{ij}.
\end{align}
To enforce the hard constraint $A \in H^1(\Omega, O_m)$,  one can employ two relaxed 
functionals parametrized by $0<\epsilon\ll 1$:
\begin{align} 
&\text{Model 1}: \quad  \min_{A\in H^1(\Omega, \mathbb R^{m\times m})}
\int_{\Omega} \left( \frac 12 \| \nabla A \|_F^2 + \frac 1{2\epsilon^2}
\mathrm{dist}^2(O_n, A) \right) \,dx;  \label{a20} \\
&\text{Model 2}: \quad  \min_{A\in H^1(\Omega, \mathbb R^{m\times m})}
\int_{\Omega} \left( \frac 12 \| \nabla A \|_F^2 + \frac 1{4\epsilon^2}
\| A^{\mathrm T} A - \mathrm{I}_m \|_F^2  \right) \,dx. \label{a21}
\end{align}
As shown in \cite{OW20}, these in turn lead to the following gradient flows
\begin{align}
&\text{Model 1}: \quad \partial_t A =  \Delta A - \epsilon^{-2}
U(\Sigma - \mathrm{I}_m) V^{\mathrm T}; \label{a22} \\
&\text{Model 2}: \quad 
\partial_t A = \Delta A - \epsilon^{-2} U (\Sigma^2 - \mathrm I_m)\Sigma V^{\mathrm T},
\end{align}
where $A=U \Sigma V^{\mathrm T}$ is the singular value decomposition of the nonsingular matrix $A$.
The gradient flow in Model 2 can be further simplified as
\begin{align} \label{a23}
\partial_t A = \Delta A - \epsilon^{-2} A (A^{\mathrm T} A - \mathrm I_m).
\end{align}
In \cite{OW20}, the authors introduced an energy-splitting method to find local minima 
of \eqref{a20} and \eqref{a21}. These are nontrivial stationary solutions other than the trivial constant orthogonal matrix-valued function) of \eqref{a22} and \eqref{a23}. The method 
can be rephrased as the following operator-splitting:
\begin{align} \label{a25}
U^{n+1} = \mathcal S_{\mathcal N}^{\mathrm{Proj} } \mathcal S_{\mathcal L}(\tau) U^n,
\end{align}
where $S_{\mathcal L}(\tau) = e^{\tau \Delta}$ is applied to the matrix entry-wise, and
\begin{align}
\Bigl( \mathcal S_{\mathcal N}^{\mathrm{Proj} } A \Bigr)(x)= U(x) V^{\mathrm T}(x),
\qquad \text{if $A(x) = U(x) \Sigma(x) V^{\mathrm T}(x)$}.
\end{align}

In this work, we take a direct approach to \eqref{a23} and employ a Strang-splitting method
to solve \eqref{a23} efficiently and accurately.  For simplicity of presentation we shall
take $\epsilon=1$ in \eqref{a23}.  We have the following theorem.

\begin{thm}[Stability for matrix-valued AC] \label{thm2}
Suppose  $\Omega =[-\pi, \pi]^d$ is the $2\pi$-periodic 
$d$-dimensional torus in physical dimensions $d\le 3$. 
 Consider the matrix-valued AC for $U:\, [0,\infty)
\times \Omega \to \mathbb R^{m\times m}$, $m\ge 1$:
\begin{align}
\begin{cases}
\partial_t U=  \Delta U + U - U U^{\mathrm T} U, 
\quad (t,x) \in (0,\infty) \times \Omega; \\
U \Bigr|_{t=0} = U^0,
\end{cases}
\end{align}
where $U^0:\, \Omega \to \mathbb R^{m\times m}$ is the initial data.
Define $\mathcal S_{\mathcal L}(t) = e^{t \Delta}$ and  for $A\in
\mathbb R^{m\times m}$, 
\begin{align}
\mathcal S_{\mathcal N}(t) A\coloneqq \left((e^{2t}-1) A A^{\mathrm T} + I\right)^{-\frac12} e^t A.
\end{align}
Define for $n\ge 0$ the Strang-splitting iterates
\begin{align}
U^{n+1}  = \mathcal S_{\mathcal L}\left( \tau/2\right) \mathcal S_{\mathcal N}\left( \tau \right)\mathcal S_{\mathcal L}\left( \tau/2\right) U^n.
\end{align}
The following hold.
\begin{enumerate}
\item \underline{The maximum principle}. For any $\tau>0$ and any $n\ge 0$, it holds that
\begin{align}
\|  \|U^{n+1}\|_F  \|_{L_x^\infty} \le \max \{ \sqrt m, \; \| \| U^n\|_F \|_{L_x^\infty} \}.
\end{align}
It follows that
\begin{align}
\sup_{n\ge 1} \| \| U^n\|_F \|_{L_x^\infty} \le \max\{\sqrt m, \; \| \| U^0\|_F \|_{L_x^\infty} \}.
\end{align}
In  particular if $\| \|U^0\|_F \|_{L_x^\infty} \le \sqrt m$, then
\begin{align}
\sup_{n\ge 1} \| \| U^n\|_F \|_{L_x^\infty} \le \sqrt m.
\end{align}

\item \underline{Modified energy dissipation for small time}. 
Assume $\| \| U^0\|_F \|_{L_x^{\infty}} \le \sqrt m$.
If $\tau>0$ satisfies 
\begin{align}
m e^\tau (e^{2\tau}-1)\le 0.43,
\end{align}
  then 
\begin{align}
\widetilde E(\tilde {U}^{n+1})\le \widetilde E(\tilde {U}^{n}),\quad \forall\, n\ge 0.
\end{align}
Here
\begin{align}
& \tilde{U}^n = \mathcal S_{\mathcal L}\left( \tau/2\right)U^n ; \\
&\widetilde E({U}) =\int_\Omega \frac{1}{2\tau}\left\langle (e^{-\tau \Delta} -1)  {U},  {U} \right\rangle_F + \left\langle G({U}),I\right\rangle_F \, dx;\\
& G(U)  = \frac1{2\tau} U U^{\mathrm T} - \frac{e^\tau}{\tau(e^{2\tau}-1)} \left(\left(\mathrm I+(e^{2\tau}-1)U U^{\mathrm T}\right)^{\frac12}-\mathrm I\right).
\end{align}
In the above $\langle A,B \rangle_F = \mathrm{Tr}(A^{\mathrm T}B)
= \sum_{i,j} A_{ij} B_{ij} $ denotes the usual Frobenius inner product.
\end{enumerate}
\end{thm}

\begin{rem}
We should point it out that the dynamics of the matrix-valued Allen-Cahn case are in general
qualitatively different from the vector-valued Allen-Cahn case. In particular there does not
appear to exist a simple procedure such that the vector-valued
AC model can be embedded into the matrix-valued  AC model.  A very tempting idea is
to consider the following system
\begin{align}
\begin{cases}
\partial_t U = \Delta U +U  - UU^{\mathrm T} U, \\
U\Bigr|_{t=0} = U^0 =\mathbf a {\mathbf a}^{\mathrm T},
\end{cases}
\end{align}
where $\mathbf a:\, \Omega \to \mathbb R^m$.  In yet other words, we consider the matrix-valued
AC model with rank one initial data. It is natural to speculate that $U$ is connected
with the solution to
\begin{align}
\begin{cases}
\partial_t \mathbf u = \Delta \mathbf u  + \mathbf u -|\mathbf u |^2 \mathbf u, \\
\mathbf u \Bigr|_{t=0} = \mathbf a.
\end{cases}
\end{align}
However one can check that  $U \ne \mathbf u {\mathbf u}^{\mathrm T}$ for $t>0$. 
The main reason is that 
\begin{align}
e^{t\Delta} \Bigl( \mathbf a {\mathbf a}^{\mathrm T}  \Bigr)
\ne {e^{t\Delta} \mathbf a } ( {e^{t\Delta} \mathbf a} )^{\mathrm T}.
\end{align}
If one drops the Laplacian and adopt only the nonlinear evolution, then one can show that $U = {\mathbf u} {\mathbf u}^{\mathrm T} $.
\end{rem}

The rest of this article is organized as follows. In Section 2 we carry out the proof of 
Theorem \ref{thm1}. In Section 3 we analyze the Strang-splitting for the matrix-valued
Allen-Cahn equation. In Section 4 we showcase a few numerical simulations for the
vector-valued Allen-Cahn and the matrix-valued Allen-Cahn equations. 

\section{Vector-valued Allen-Cahn}\label{sect2}
In this section we give the proof of Theorem \ref{thm1}. We consider the vector-valued Allen-Cahn equation for
$\mathbf u=\mathbf u(t, x): \, [0,\infty) \times \Omega \to \mathbb R^{m}$, $m\ge 1$:
\begin{align}\label{eq:2.1}
\begin{cases}
\partial_t \mathbf u = \Delta \mathbf u + \mathbf u - |\mathbf u|^2 \mathbf u,
\qquad (t, x ) \in (0,\infty) \times \Omega;  \\
\mathbf u \Bigr|_{t=0} =\mathbf u^0, \quad x \in \Omega.
\end{cases}
\end{align}
Here $|\mathbf u|^2$ is the usual $l^2$ norm, i.e. $|\mathbf u|^2 = \sum_{i=1}^m
 u_i^2$ for $\mathbf u = (u_1, \cdots, u_m)^{\mathrm T}$. 
The spatial domain $\Omega=[-\pi, \, \pi]^d$ is the $2\pi$-periodic torus in physical
dimensions $d\le 3$.  

We proceed 
in several steps. 

\subsection{Properties of $\mathcal S_{\mathcal L}$ and $ \mathcal S_{\mathcal N}$}
We first consider the pure nonlinear evolution. This is driven by the following ODE system
written for $\mathbf u = \mathbf u (t): \, [0,\infty) \to \mathbb R^m$.
\begin{align} \label{s1}
\begin{cases}
\frac d {dt} {\mathbf u} = (1-|\mathbf u|^2) \mathbf u, \\
\mathbf u \Bigr|_{t=0} =\mathbf a \in \mathbb R^m.
\end{cases}
\end{align}
\begin{prop}[The explicit nonlinear propagator $\mathcal S_{\mathcal N}(t)$]
Given $\mathbf a\in \mathbb R^{m}$, the unique smooth
 solution $U(t)$ to \eqref{s1} is given by 
\begin{align} \label{s2}
\mathcal S_{\mathcal N}(t) \mathbf a:= U(t) = \left((e^{2t}-1) |\mathbf a|^2+1\right)^{-\frac12} e^t 
\mathbf a,\qquad t>0.
\end{align}
\end{prop}
\begin{rem}
If $\mathbf a$ is a vector-valued function, i.e. $\mathbf a:\, \Omega \to \mathbb R^{m}$, then
we naturally extend the definition of $\mathcal S_{\mathcal N}(t) \mathbf a$ as
\begin{align}
\Bigl(  \mathcal S_{\mathcal N} (t) \mathbf a \Bigr)(x) = \mathcal S_{\mathcal N} (t) (\mathbf a(x)),
\quad x \in \Omega.
\end{align}
This convention will be used without explicit mentioning.
\end{rem}
\begin{proof}
Taking the $l^2$-inner product on both sides of \eqref{s1} gives us 
\begin{align}
\frac 12 \frac d {dt}  |{\mathbf u}|^2 = (1-|\mathbf u|^2) |{\mathbf u}|^2.
\end{align}
This is an ODE for $|\mathbf u|^2$ which has an explicit solution:
\begin{align} \label{s3}
|\mathbf u(t)|^2 = \frac{e^{2t} |\mathbf a|^2}{(e^{2t}-1)|\mathbf a|^2+1}.
\end{align}
Plugging the above into \eqref{s1}, we obtain
\begin{align}
\frac d {dt} {\mathbf u} = \frac{1-|\mathbf a|^2}{(e^{2t}-1)|\mathbf a|^2+1}\mathbf u.
\end{align}
It is not difficult to work out the explicit solution as
\begin{align}
\mathbf u(t) = \frac{e^t\mathbf a}{\left((e^{2t}-1)|\mathbf a|^2+1\right)^{\frac12}}.
\end{align}
\end{proof}

Given $\mathbf u=(u_1,\cdots, u_m)^{\mathrm T}:\, \Omega \to \mathbb R^{m}$ and $t>0$, we define 
the linear propagator 
\begin{align}
\Bigl( \mathcal S_{\mathcal L} (t) \mathbf u \Bigr)_{i} (x) = ( e^{t\Delta} u_{i} )(x),
\qquad i=1,\cdots,m.
\end{align}
In yet other words, the operator $\mathcal S_{\mathcal L}(t) = e^{t\Delta}$ is applied to the 
vector $\mathbf u$ entry-wise.

\begin{thm}[Maximum principle for $\mathcal S_{\mathcal L}$ and $\mathcal S_{\mathcal N}$]
\label{s5}
Let $\Omega=[-\pi, \pi]^d$ be the $2\pi$-periodic $d$-dimensional torus. 
For any $\tau>0$,  the following hold.
\begin{enumerate}
\item For any measurable vector-valued $\mathbf a:\, \Omega \to \mathbb R^{m}$, we have
\begin{align} \label{s6}
& \|  | \mathcal S_{\mathcal L}(\tau) \mathbf a | \|_{L_x^{\infty} } \le 
 \| | \mathbf a(x) | \|_{L_x^{\infty} }.
\end{align}

\item  For any $\mathbf w\in \mathbb R^{m}$, we have
\begin{align} \label{s7}
| \mathcal S_{\mathcal N}(\tau) \mathbf w | \le \max\{1,\, |\mathbf w| \}.
\end{align}
\end{enumerate}
\end{thm}
\begin{proof}
We first show \eqref{s6}. Clearly for any vector $\mathbf v \in \mathbb R^m$, we have
\begin{align} \label{s7.1}
|\mathbf v | = \sup_{\tilde {\mathbf v} \in \mathbb R^{ m}:\;
| \tilde {\mathbf v}|\le 1 } \langle \mathbf v, \, \mathbf {\tilde v} \rangle,
\end{align}
where $\langle, \rangle$ denotes the usual $l^2$-inner product.
With no loss we may assume $\| | \mathbf a(x) | \|_{L_x^{\infty} } \le 1$. Fix $x_0 \in \Omega$.
It suffices for us to show 
\begin{align}
  \left| \int_{\Omega} k(x_0-y) \mathbf a(y) dy \right|\le 1,
  \end{align}
where $k(\cdot)$ is the scalar-valued kernel corresponding to $\mathcal S_{\mathcal L}(\tau)$. 
By \eqref{s7.1}, we only need to check for any $\tilde {\mathbf v}$ with $|\tilde 
{\mathbf v} | \le 1$, 
\begin{align}
\int_{\Omega} k(x_0-y) \langle \mathbf a(y),\, \tilde {\mathbf v} \rangle dy \le 1.
\end{align}
But this is obvious since $\int_{\Omega} k(x_0-y) dy =1$ and $
\langle \mathbf a(y), \tilde {\mathbf v}\rangle  \le \| 
| \mathbf a(x) | \| \|_{L_x^{\infty}} |\tilde {\mathbf v} | \le 1$. 

We turn now to \eqref{s7}. By \eqref{s3}, we have
\begin{align} 
|\mathcal S_{\mathcal N}(t) \mathbf w|^2 = \frac{e^{2t} |\mathbf w|^2}{(e^{2t}-1)|\mathbf w|^2+1}.
\end{align}
Consider the scalar function 
\begin{align}
\varphi (\lambda) = \frac{e^{2t} \lambda} { (e^{2t}-1) \lambda +1}.
\end{align}
It is not difficult to check that $\varphi $ is monotonically increasing on $[0, \infty)$.
Furthermore if $\lambda \ge 1$, then 
\begin{align}
\varphi(\lambda) \le \frac{e^{2t} \lambda} {(e^{2t}-1)+1} =\lambda.
\end{align}
The desired result clearly follows.
\end{proof}

\begin{lem} \label{s8}
Let $\tau>0$ and consider $G:\, \mathbb R^m \to \mathbb R$ defined as
\begin{align}
G(\mathbf u)  = \frac1{2\tau} |\mathbf u|^2 - \frac{e^\tau}{\tau(e^{2\tau}-1)} \left(\left(
1+(e^{2\tau}-1)|\mathbf u|^2\right)^{\frac12}-1\right).
\end{align}
For any $\mathbf u, \mathbf v \in \mathbb R^m$, it holds that
\begin{align}
- \langle (\nabla G)(\mathbf u), \, \mathbf v- \mathbf u \rangle 
\le G(\mathbf u) -G(\mathbf v) + \frac 1 {2\tau} | \mathbf v -\mathbf u|^2.
\end{align}
Here 
$\langle \mathbf a, \, \mathbf b \rangle = \sum_{i=1}^m a_i b_i$
for $\mathbf a = (a_1,\cdots, a_m)^{\mathrm T}$,
$\mathbf b =(b_1,\cdots, b_m)^{\mathrm T} \in \mathbb R^m$.
\end{lem}
\begin{proof}
We first examine the auxiliary function
\begin{align}
h(\mathbf u) = - (1+ |\mathbf u |^2)^{\frac 12}.
\end{align}
Clearly
\begin{align}
 \partial_i h  &= - (1+|\mathbf u|^2)^{-\frac 12} u_i;  \\
 \partial_{ij} h &=
 (1+|\mathbf u|^2)^{-\frac 32} u_i u_j - (1+|\mathbf u|^2)^{-\frac 12} \delta_{ij}
 \\
 & = (1+|\mathbf u|^2)^{-\frac 12}
 \Bigl( \frac {u_i} {\sqrt{1+|\mathbf u|^2} }
 \frac {u_j}{ \sqrt{1+|\mathbf u|^2} } -\delta_{ij} \Bigr).
 \end{align}
 Thus 
 \begin{align}
 \sum_{i,j=1}^m \xi_i \xi_j  \partial_{ij} h  \le 0,
 \qquad\forall\, \mathbf {\xi} =(\xi_1,\cdots, \xi_m)^{\mathrm T} \in
 \mathbb R^m, \;\forall\, \mathbf u \in \mathbb R^m.
 \end{align}
 In yet other words, the function $h$ is concave. Our desired result then easily follows from
 Taylor expanding $G(\mathbf u +\theta (\mathbf v - \mathbf u) )$ for $\theta
 \in [0,1]$. 
\end{proof}

\begin{thm}[Unconditional modified energy dissipation for vector-valued Allen-Cahn]
Suppose  $\Omega =[-\pi, \pi]^d$ is the $2\pi$-periodic 
$d$-dimensional torus in physical dimensions $d\le 3$. 
Let  $\mathbf u^0: \; \Omega \to  \mathbb R^{m}$ satisfy $ \| |\mathbf u^{0}(x)| \| \|_{L^{\infty}_x} \le 1$. 
Recall $\mathcal S_{\mathcal L}(t) = e^{t \Delta}$ and  for $\mathbf w\in
\mathbb R^{m}$, 
\begin{align}
\mathcal S_{\mathcal N}(t) \mathbf w \coloneqq 
\left((e^{2t}-1) |\mathbf w|^2+1\right)^{-\frac12} e^t
\mathbf w.
\end{align}
Define for $n\ge 0$ the Strang-splitting iterates
\begin{align}
\mathbf u^{n+1}  = \mathcal S_{\mathcal L}\left( \tau/2\right) \mathcal S_{\mathcal N}\left( \tau \right)\mathcal S_{\mathcal L}\left( \tau/2\right) \mathbf u^n.
\end{align}
For any $\tau>0$, we have
\begin{align}
\widetilde E(\tilde {\mathbf u}^{n+1})\le \widetilde E(\tilde {\mathbf u}^{n}),\quad \forall\, n\ge 0,
\end{align}
where 
\begin{align}
& \tilde{\mathbf u}^n = \mathcal S_{\mathcal L}\left( \tau/2\right)\mathbf u^n ; \\
&\widetilde E({\mathbf u}) =\int_\Omega \left( \frac{1}{2\tau}\left\langle (e^{-\tau \Delta} -1)  {\mathbf u},  {\mathbf u} \right\rangle + G({\mathbf u})  \right)\, dx;\\
& G(\mathbf u)  = \frac1{2\tau} |\mathbf u|^2 - \frac{e^\tau}{\tau(e^{2\tau}-1)} \left(\left(
1+(e^{2\tau}-1)|\mathbf u|^2\right)^{\frac12}-1\right).
\end{align}
In the above $\langle \mathbf a, \, \mathbf b \rangle = \sum_{i=1}^m a_i b_i$
for $\mathbf a = (a_1,\cdots, a_m)^{\mathrm T}$,
$\mathbf b =(b_1,\cdots, b_m)^{\mathrm T} \in \mathbb R^m$.
\end{thm} 
\begin{proof}
By definition we have
\begin{align}
e^{-\tau \Delta}\tilde{\mathbf u}^{n+1}  = \frac{e^\tau\tilde{\mathbf u}^n}{\left((e^{2\tau}-1)|\tilde{\mathbf u}^n|^2+1\right)^{\frac12}}.
\end{align}
We rewrite it as
\begin{align} \label{s10}
\frac 1 {\tau} (e^{-\tau \Delta}-1)\tilde{\mathbf u}^{n+1} +\frac 1 {\tau}  (\tilde{\mathbf u}^{n+1}-\tilde{\mathbf u}^{n}) = \frac 1 {\tau} \left( \frac{e^\tau\tilde{\mathbf u}^n}{\left((e^{2\tau}-1)|\tilde{\mathbf u}^n|^2+1\right)^{\frac12}}-\tilde{\mathbf u}^{n}
\right).
\end{align}
Note that 
\begin{align}
( \nabla G)({\mathbf u})  =\frac{\mathbf u}\tau -\frac{e^\tau{\mathbf u}}{\tau\left((e^{2\tau}-1)|{\mathbf u}|^2+1\right)^{\frac12}}. 
\end{align}
By Lemma \ref{s8}, we have
\begin{align}
& \left\langle \frac 1 {\tau} \Bigl( \frac{e^\tau\tilde{\mathbf u}^n}{\left((e^{2\tau}-1)|\tilde{\mathbf u}^n|^2+1\right)^{\frac12}}-\tilde{\mathbf u}^{n}
\Bigr),  \; \tilde {\mathbf u}^{n+1} - \tilde {\mathbf u}^{n} \right\rangle \\
\le & \; G( \tilde {\mathbf u}^{n} ) -G( \tilde {\mathbf u}^{n+1})
+ \frac{1}{2\tau} |\tilde {\mathbf u}^{n+1} - \tilde {\mathbf u}^{n}|^2.
\end{align}
Taking the $l^2$-inner product with   $  (\tilde {\mathbf u}^{n+1} - \tilde {\mathbf u}^{n})$ 
and integrating in $x$ on both sides of \eqref{s10}, we have 
\begin{equation}
\begin{aligned}
\widetilde E(\tilde {\mathbf u}^{n+1})- \widetilde E(\tilde {\mathbf u}^{n}) 
&\le -\frac{1}{2\tau} \int_{\Omega} |\tilde {\mathbf u}^{n+1} - \tilde {\mathbf u}^{n}|^2 \,dx
\le 0.
\end{aligned}
\end{equation}
\end{proof}

\section{Matrix-valued Allen-Cahn equation}\label{sect3}
In this section we carry out  the proof of Theorem \ref{thm2} in several steps. 
We study the matrix-valued Allen-Cahn equation for
$U=U(t, x): \, [0,\infty) \times \Omega \to \mathbb R^{m\times m}$:
\begin{align} 
\partial_t U = \Delta U  + U - UU^{\mathrm T} U.
\end{align}
The spatial domain $\Omega=[-\pi, \, \pi]^d$ is the $2\pi$-periodic torus in physical
dimensions $d\le 3$.  

\subsection{Definition and properties of $\mathcal S_{\mathcal N}$ and $\mathcal S_{\mathcal L}$}
We consider first the pure nonlinear part, i.e. the following 
ODE for $U=U(t):\, [0,\infty) \to \mathbb R^{m\times m}$:
\begin{align} \label{3t0}
\begin{cases}
\frac d {dt} U = U-UU^{\mathrm T}U, \\
U \Bigr|_{t=0} = U_0 \in \mathbb R^{m\times m}.
\end{cases}
\end{align}
Remarkably, we find that the above ODE admits an explicit solution. 
\begin{prop}[The explicit nonlinear propagator $\mathcal S_{\mathcal N}(t)$]
Given $U_0 \in \mathbb R^{m\times m}$, the unique smooth
 solution $U(t)$ to \eqref{3t0} is given by 
\begin{align} \label{3t1}
\mathcal S_{\mathcal N}(t) U_0 := U(t) = \left((e^{2t}-1) U_0U_0^{\mathrm T} + \mathrm I\right)^{-\frac12} e^t U_0,\qquad t>0.
\end{align}
\end{prop}
\begin{rem}
If $U_0$ is a matrix-valued function, i.e. $U_0:\, \Omega \to \mathbb R^{m\times m}$, then
we naturally extend the definition of $\mathcal S_{\mathcal N}(t) U_0$ as
\begin{align}
\Bigl(  \mathcal S_{\mathcal N} (t) U_0 \Bigr)(x) = \mathcal S_{\mathcal N} (t) (U_0(x)),
\quad x \in \Omega.
\end{align}
This convention will be used without explicit mentioning.
\end{rem}
\begin{proof}
We begin by noting that
\begin{align}
\frac d {dt} \Bigl(   (e^{2t}-1) U_0 U_0^{\mathrm T} + \mathrm I \Bigr)
= 2e^{2t} U_0 U_0^{\mathrm T}.
\end{align}
This clearly commutes with $  (e^{2t}-1) U_0 U_0^{\mathrm T} + \mathrm I$.  In particular we
have
\begin{align}
\frac d {dt}
\Bigl(  ( (e^{2t}-1) U_0 U_0^{\mathrm T} + \mathrm I)^{-\frac 12} \Bigr)
&=  - e^{2t} U_0 U_0^{\mathrm T}  
( (e^{2t}-1) U_0 U_0^{\mathrm T} + \mathrm I)^{-\frac 32} \\
&= - ( (e^{2t}-1) U_0 U_0^{\mathrm T} + \mathrm I)^{-\frac 32}  
e^{2t} U_0 U_0^{\mathrm T}.
\end{align}
With the above we obtain 
\begin{align}
 U^{\prime}(t) = \left((e^{2t}-1) U_0U_0^{\mathrm T} + \mathrm I\right)^{-\frac12} e^t U_0
-\left((e^{2t}-1) U_0U_0^{\mathrm T} + \mathrm I\right)^{-\frac32} e^{3t} U_0 U_0^{\mathrm T} U_0.
\end{align}
Note that $\left((e^{2t}-1) U_0U_0^{\mathrm T} + \mathrm I\right)^{-\frac12}$ and $U_0U_0^{\mathrm T}$ commute.  It follows that 
\begin{equation}
\begin{aligned}
U U^{\mathrm T} U 
& = \left((e^{2t}-1) U_0U_0^{\mathrm T} + \mathrm I\right)^{-\frac12} e^{2t} U_0 U_0^{\mathrm T} \left((e^{2t}-1) U_0U_0^{\mathrm T} + \mathrm I\right)^{-1}e^t U_0 \\
& = \left((e^{2t}-1) U_0U_0^{\mathrm T} +\mathrm I\right)^{-\frac32} e^{3t} U_0 U_0^{\mathrm T} U_0.
\end{aligned}
\end{equation}
Therefore, $U(t)$ satisfies 
\begin{equation}
\frac d {dt} U = U-UU^{\mathrm T}U.
\end{equation}
\end{proof}

Given $U:\, \Omega \to \mathbb R^{m\times m}$ and $t>0$, we define 
the linear propagator 
\begin{align}
\Bigl( \mathcal S_{\mathcal L} (t) U \Bigr)_{ij} (x) =  \Bigl( e^{t\Delta} U_{ij} \Bigr)(x),
\qquad i, j=1,\cdots m.
\end{align}
In yet other words, the operator $\mathcal S_{\mathcal L}(t) = e^{t\Delta}$ is applied to the matrix $U$ entry-wise.

\begin{thm}[Maximum principle for $\mathcal S_{\mathcal L}$ and $\mathcal S_{\mathcal N}$]
\label{3t5}
Let $\Omega=[-\pi, \pi]^d$ be the $2\pi$-periodic $d$-dimensional torus. 
For any $\tau>0$,  the following hold.
\begin{enumerate}
\item For any measurable matrix-valued $A:\, \Omega \to \mathbb R^{m\times m}$, we have
\begin{align} \label{2t1}
& \|  \| \mathcal S_{\mathcal L}(\tau) A \|_F \|_{L_x^{\infty} } \le 
 \| \| A(x) \|_F \|_{L_x^{\infty} }.
\end{align}

\item  For any $B \in \mathbb R^{m\times m}$ with $\| B\|_F \le \sqrt m$, we have
\begin{align} \label{2t2}
\| \mathcal S_{\mathcal N}(\tau) B \|_F \le \sqrt m.
\end{align}
\end{enumerate}
\end{thm}
\begin{rem}
In  \cite[Prop. 3.2.]{OW20}, Osting and Wang proved a maximum principle for $\mathcal S_{\mathcal L}(\tau)$ under the assumption that $A$ is a continuous function with $\|A\|_F =1$ for every
$x \in \Omega$. We do not need such a stringent assumption here. Our result here is optimal and the proof appears to be simpler.
\end{rem}

\begin{proof}
We first show \eqref{2t1}. Recall the usual Frobenius inner product:
\begin{align}
\langle M_1, \, M_2 \rangle_F = \sum_{i,j=1}^m (M_1)_{ij} (M_2)_{ij}
=\mathrm{Tr} (M_1 M_2^{\mathrm T} ).
\end{align}
For any matrix $M \in \mathbb R^{m\times m}$, we clearly  have
\begin{align} \label{2t3}
\| M \|_F = \sup_{\tilde M \in \mathbb R^{m\times m}:\;
\| \tilde M \|_F \le 1} \langle M, \, \tilde M \rangle_F.
\end{align}
With no loss we may assume $\| \| A(x) \|_F \|_{L_x^{\infty} } \le 1$. Fix $x_0 \in \Omega$.
It suffices for us to show 
\begin{align}
  \| \int_{\Omega} k(x_0-y) A(y) dy \|_F \le 1,
  \end{align}
where $k(\cdot)$ is the scalar-valued kernel corresponding to $\mathcal S_{\mathcal L}(\tau)$. 
By \eqref{2t3}, we only need to check for any $\tilde M$ with $\| \tilde M \|_F \le 1$, 
\begin{align}
\int_{\Omega} k(x_0-y) \langle A(y),\, \tilde M \rangle_F dy \le 1.
\end{align}
But this is obvious since $\int_{\Omega} k(x_0-y) dy =1$ and $
\langle A(y), \tilde M \rangle_F \le \| \|A(x) \|_F \|_{L_x^{\infty}} \|\tilde M \|_F \le 1$. 

Next we show \eqref{2t2}.  Denote $U=U(t) = \mathcal S_{\mathcal N}(t) B$. Clearly
\begin{align}
\partial_t U = U - UU^{\mathrm T} U.
\end{align}
Taking the $L^2$ Frobenius inner with $U$ on both sides of the above equation, we obtain
\begin{align}
\frac 12 \partial_t  \alpha(t) = \alpha(t) -  \| U(t) U(t)^{\mathrm T} \|_F^2,
\end{align}
where we have denoted 
\begin{align}
\alpha(t) = \langle U(t), \, U(t) \rangle_F = \mathrm{Tr} ( U(t) U(t)^{\mathrm T} ).
\end{align}
Note that 
\begin{align}
\alpha=\mathrm{Tr}(UU^{\mathrm T} ) = \langle UU^{\mathrm T}, \mathrm I \rangle_F
\le \| UU^{\mathrm T} \|_F \sqrt m.
\end{align}
Thus 
\begin{align}
\| U(t) U(t)^{\mathrm T} \|_F^2 \ge \frac 1 m \alpha(t)^2.
\end{align}
It follows that 
\begin{align}
\frac 12 \partial_t \left( \frac 1 m \alpha(t) \right) \le \frac 1m \alpha(t) - \left( \frac 1m \alpha(t) \right)^2.
\end{align}
It is not difficult to check  that $\frac 1m \alpha(t) $ is a continuously-differentiable function of $t$ defined
for all $t\ge 0$, nonnegative and
$\frac 1 m \alpha(0) \le 1$.  By a simple argument-by-contradiction, we can show that
for any $\delta_1>0$
\begin{align}
\sup_{t\ge 0} \frac 1 m \alpha(t) \le 1+\delta_1.
\end{align}
Sending $\delta_1$ to zero then yields the desired estimate.
\end{proof}

\begin{rem}
An alternative proof of \eqref{2t2} goes as follows. 
Since $\|B\|_F\le \sqrt m$,  we have
\begin{equation}
\mathrm{Tr}(BB^{\mathrm T}) = \sum_{i=1}^{m} \lambda_i\le m,
\end{equation}
where $\lambda_i\ge 0$ are the eigenvalues of $BB^{\mathrm T}$.
By \eqref{3t1}, we have
\begin{align} 
\mathcal S_{\mathcal N}(t) B &= \left((e^{2t}-1) BB^{\mathrm T} + \mathrm I\right)^{-\frac12} e^t B ;
\\
 \|\mathcal S_{\mathcal N}(t) B \|_F^2 
&=e^{2t}  \mathrm{Tr} \left(  ( (e^{2t}-1) BB^{\mathrm T} +
\mathrm I  )^{-\frac 12}  B B^{\mathrm T} ( (e^{2t}-1) BB^{\mathrm T} +
\mathrm I  )^{-\frac 12}
\right) \notag \\
& = e^{2t} \mathrm{Tr} \left(  ( (e^{2t}-1) BB^{\mathrm T} +
\mathrm I  )^{-1}  B B^{\mathrm T} \right) \\
& =   \sum_{i=1}^{m} 
\underbrace{e^{2t}\left((e^{2t}-1) \lambda_i+ 1\right)^{-1} \lambda_i }_{=:\varphi(\lambda_i)}.
\end{align}
Clearly for any $\lambda \ge 0$, 
\begin{align}
\varphi'(\lambda) &=e^{2t} \left((e^{2t}-1) \lambda+ 1\right)^{-2}\ge 0; \\
\varphi''(\lambda) & = -2 e^{2t} (e^{2t}-1) \left((e^{2t}-1) \lambda+ 1\right)^{-3}\le 0.
\end{align}
In particular  $\varphi$ is a concave function on $[0, \infty)$.  By Jensen's inequality and the
fact that $\frac 1 m \sum_{i=1}^m \lambda_i \le 1$ , we have
\begin{align}
\frac 1m \sum_{i=1}^m \varphi( \lambda_i) \le \varphi( \frac 1 m\sum_{i=1}^m \lambda_i)
\le \varphi (1)= 1.
\end{align}
Thus $ \| \mathcal S_{\mathcal N} (t) B \|_F^2 \le m $.
\end{rem}

\subsection{Modified energy dissipation}
In this subsection we shall often use (sometimes without explicit mentioning)
 the obvious identity
\begin{align}
\mathrm{Tr}(AB^{\mathrm T} ) = \langle A, \, B \rangle_F = \sum_{i,j=1}^m A_{ij} B_{ij},
\qquad\forall\, A, B \in \mathbb R^{m\times m},
\end{align}
where $\langle, \rangle_F$ denotes the usual Frobenius inner product.
In particular
\begin{align}
\mathrm{Tr}(A) = \langle A, \, \mathrm I \rangle_F.
\end{align}
It follows that if $A=A(s)$, $s\in [0,1]$ is a continuously differentiable matrix-valued function, then
\begin{align}
\frac d {ds} \mathrm{Tr}(A(s)) = \langle A^{\prime}(s), \, \mathrm I \rangle_F
= \mathrm{Tr} (A^{\prime}(s) ).
\end{align}
Other formulae follow similarly from the above identities.

\begin{lem} \label{lemt0.1}
Suppose $B=B(s):\; s\in[0,1]\to \mathbb R^{m\times m}$ is continuously differentiable with
\begin{align}
\max_{0\le s \le 1} \| B(s) \|_F <1.
\end{align}
For any $s \in [0, 1]$ and any $B_1 \in \mathbb R^{m\times m}$, it holds that
\begin{align}
\left| \mathrm{Tr}\Bigl( \frac d {ds} \left( (\mathrm I + B(s) )^{-\frac 12} \right) B_1 \Bigr)
\right|
\le   \frac 12 (1- \|B(s) \|_F)^{-\frac 32} \| B^{\prime}(s) \|_F \| B_1 \|_F.
\end{align}
\end{lem}
\begin{proof}
It suffices for us to bound $\| \frac d {ds} \Bigl( (\mathrm I + B(s) )^{-\frac 12} \Bigr) \|_F$.
Recall the power series expansion for a real number $|x|<1$
\begin{align}
&(1-x)^{-\frac 12}  = \sum_{k\ge 0} C_k x^k, \\
&  \frac 12 (1-x)^{-\frac 32} = \sum_{k\ge 1} C_k k x^{k-1}.
\end{align}
where the coefficients $C_k$ are all positive.  For integer $k\ge 1$, we note that
\begin{align}
\frac d {ds} \left( B^k \right) = B^{\prime} B^{k-1} + B B^{\prime} B\cdots B
+BBB^{\prime} B\cdots B+\cdots+B^{k-1} B^{\prime}.
\end{align}
In particular we do not assume the matrix $B^{\prime}$ commutes with $B$. On the other hand,
since the matrix Frobenius norm
is sub-multiplicative, we have
\begin{align}
\| \frac d {ds} \left( B^k \right)\|_F
\le k \|B\|_F^{k-1} \| B^{\prime} \|_F.
\end{align}
It follows that
\begin{align}
 & \left\| \frac d {ds} \Bigl(  (\mathrm I + B(s) )^{-\frac 12} \Bigr) \right\|_F \\
 \le & \sum_{k\ge 1} C_k \Bigl\| \frac d {ds}  ( B(s)^k ) \Bigr\|_F \\
 \le & \sum_{k\ge 1}  C_k k \|B(s)\|_F^{k-1} \| B^{\prime}(s) \|_F \\
 = & \frac 12 (1- \| B(s) \|_F)^{-\frac 32} \| B^{\prime}(s) \|_F.
 \end{align}
The desired result then easily follows.
\end{proof}

\begin{lem} \label{lemt0.2}
Denote by $\mathbb R^{m\times m}_{\mathrm{sp}}$ the set of symmetric positive-definite matrices in $\mathbb R^{m\times m}$.
Suppose $B=B(s):\; s\in[0,1]\to \mathbb R^{m\times m}_{\mathrm{sp}}$ is continuously differentiable with 
\begin{align} \label{t0.2a}
\xi^{\mathrm T} B(s) \xi \ge \eta_1>0, \qquad\forall\, \xi \in \mathbb R^m, \; \forall\, s \in [0,1].
\end{align}
Then 
\begin{align} \label{t0.2a0}
\frac d {ds} \mathrm{Tr} ( B(s)^{\frac 12} ) = \frac 12 \mathrm{Tr}
\left( B(s)^{-\frac 12} B^{\prime}(s) \right).
\end{align}
\end{lem}
\begin{rem}
Later we shall take $B(s) = \mathrm I + \epsilon \phi(s) \phi(s)^{\mathrm T} $ with
$\epsilon>0$ sufficiently small and $ \phi (s) \in \mathbb R^{m\times m}$. In that case we
can directly make use of the power series expansion and derive \eqref{t0.2a0} for small
$\epsilon$. The strength of Lemma \ref{lemt0.2} is that the smallness of $\epsilon$ is not
needed.
\end{rem}
\begin{proof}
We begin by noting that for any integer $k\ge 2$, 
\begin{align}
\frac d {ds} \mathrm{Tr}(B(s)^k) & = \mathrm{Tr}\left( B^{\prime} B^{k-1}
+BB^{\prime} B\cdots B +\cdots+ B^{k-1} B^{\prime}\right) \\
&= k \mathrm{Tr} \left(B(s)^{k-1} B^{\prime} (s)\right).
\end{align}
It follows that for any $\alpha_0 \in \mathbb R$, 
\begin{align}
\frac d {ds} \mathrm{Tr}\left(e^{\alpha_0 B(s) } \right)
= \alpha_0 \mathrm{Tr} \left( e^{\alpha_0 B(s) } B^{\prime}(s)  \right).
\end{align}
Note that 
\begin{align} \label{t0.2b}
B(s)^{\frac 12} = \frac 1 {\Gamma(\frac 12)} \int_0^{\infty}
e^{-t B(s)} t^{-\frac 12} dt,
\end{align}
where $\Gamma(\cdot)$ is the usual Gamma function.
In view of the strict positivity assumption \eqref{t0.2a}, the convergence in 
\eqref{t0.2b} is out of question. Clearly
\begin{align}
\mathrm{Tr}( B(s)^{\frac 12}) 
= \frac 1 {\Gamma(\frac 12)} \int_0^{\infty}
\mathrm{Tr}\left( e^{-t B(s)}  \right)t^{-\frac 12} dt.
\end{align}
The desired result then easily follows.
\end{proof}

\begin{lem} \label{lemt0}
Let $U_0$, $H \in \mathbb R^{m\times m}$ satisfy $\| U_0\|_F \le \sqrt m$ and $\| U_0+H\|_F\le
\sqrt m$.  Let $\tau>0$. For $s \in [0,1]$, define
\begin{align}
&\phi= \phi(s) =U_0 + s H;  \\
& h(s) = \mathrm{Tr} \Bigl(
\frac 1 {2\tau} \phi \phi^{\mathrm T}
-\frac{e^{\tau} }{\tau (e^{2\tau}-1) }
\bigl(  (\mathrm{I} + (e^{2\tau}-1) \phi \phi^{\mathrm T} )^{\frac 12} - \mathrm {I} \bigr)
\Bigr).
\end{align}
We have
\begin{align}
& h^{\prime}(0)
= \frac 1 {\tau} \mathrm{Tr} (U_0 H^{\mathrm T} )
-\frac{e^{\tau}} {\tau} \mathrm{Tr}
\Bigl( ( \mathrm I + (e^{2\tau}-1) U_0 U_0^{\mathrm T} )^{-\frac 12} U_0 
H^{\mathrm T} \Bigr).  \label{t1}
\end{align}
If $e^{\tau} (e^{2\tau}-1) m \le \epsilon_0 <1$, then
\begin{align}
&\max_{0\le s\le 1} h^{\prime\prime}(s) \le  
\frac 1 {\tau} \Bigl( 1+ (1-\epsilon_0)^{-\frac 32} \epsilon_0 \Bigr) \| H \|_F^2. \label{t2}
\end{align}
If $e^{\tau} (e^{2\tau}-1) m \le \epsilon_0$ and $(1-\epsilon_0)^{-\frac 32} 
\epsilon_0\le 1$, then
\begin{align}
-h^{\prime}(0) \le h(0) - h(1)  +\frac 1 {\tau} \| H\|_F^2. \label{t3}
\end{align}
\begin{rem}
If we take $\epsilon_0=0.43$, then
\begin{align}
(1-\epsilon_0)^{-\frac 32} \epsilon_0 \approx 0.99209<1.
\end{align}
\end{rem}

\end{lem}
\begin{proof}
Observe that for all $s\in [0,1]$
\begin{align}
\| \phi (s) \|_F = \| s (U_0+H) + (1-s) U_0 \|_F \le \sqrt m.
\end{align}

{ By Lemma \ref{lemt0.2}, we have}
\begin{align}
 h^{\prime}(s) &
=\mathrm{Tr}\Bigl( 
\frac 1{2\tau} \bigl(  \phi^{\prime} \phi^{\mathrm{T}}  + \phi (\phi^{\prime})^{\mathrm T}
\bigr)
- \frac {e^{\tau}} {\tau (e^{2\tau} -1) }
\cdot \frac 12 ( \mathrm{I} + (e^{2\tau}-1) \phi \phi^{\mathrm T} )^{-\frac 12}
(e^{2\tau}-1) ( \phi^{\prime} \phi^{\mathrm T} + \phi  (\phi^{\prime} )^{\mathrm T} )
\Bigr) \\
& =\mathrm{Tr}\Bigl( 
\frac 1{\tau} \phi H^{\mathrm T}
- \frac {e^{\tau}} {\tau  }
\cdot \frac 12 ( \mathrm{I} + (e^{2\tau}-1) \phi \phi^{\mathrm T} )^{-\frac 12}
( H \phi^{\mathrm T} + \phi  H^{\mathrm T} )
\Bigr).
\end{align}
The equality \eqref{t1} follows from the fact that if $A\in \mathbb R^{m\times m}$ is symmetric,
then 
\begin{align}
\mathrm{Tr} (A B^T) = \mathrm{Tr} (BA)=\mathrm{Tr} (AB), \qquad\forall\, B \in 
\mathbb R^{m\times m}.
\end{align}

By direction computation, we also have
\begin{align}
h^{\prime\prime}(s)
& = \mathrm{Tr}( \frac 1 {\tau} H H^{\mathrm T} )
- \frac {e^{\tau}} {\tau} \mathrm{Tr} \Bigl(
(\mathrm{I} + (e^{2\tau}-1) \phi \phi^{\mathrm T} )^{-\frac 12} H H^{\mathrm T} \Bigr) \\
& \qquad - \frac{e^{\tau} } {2\tau}
\mathrm{Tr} \left( \frac d {ds}  \Bigl( 
( \mathrm{I} + (e^{2\tau}-1) \phi \phi^{\mathrm T} )^{-\frac 12} \Bigr)
( H \phi^{\mathrm T} + \phi  H^{\mathrm T} ) \right).
\end{align}
Clearly
\begin{align}
&\mathrm{Tr} \Bigl(
(\mathrm{I} + (e^{2\tau}-1) \phi \phi^{\mathrm T} )^{-\frac 12} H H^{\mathrm T} \Bigr)  \\
=& \mathrm{Tr} \Bigl( H^{\mathrm T}
(\mathrm{I} + (e^{2\tau}-1) \phi \phi^{\mathrm T} )^{-\frac 12} H \Bigr) \ge 0.
\end{align}
Note that 
\begin{align}
(e^{2\tau}-1)\| \phi \phi^{\mathrm T} \|_F \le 
(e^{2\tau}-1) m  <\epsilon_0<1.
\end{align}
By Lemma \ref{lemt0.1} we have
\begin{align}
& \frac {e^{\tau}} {2\tau} \mathrm{Tr} \left( \frac d {ds}  \Bigl( 
( \mathrm{I} + (e^{2\tau}-1) \phi \phi^{\mathrm T} )^{-\frac 12} \Bigr)
( H \phi^{\mathrm T} + \phi  H^{\mathrm T} ) \right)\\
\le &\; \frac{e^{\tau} } {2\tau} \cdot \frac 12 
\Bigl( 1- (e^{2\tau}-1) \| \phi \phi^{\mathrm T} \|_F \Bigr)^{-\frac 32}
(e^{2\tau}-1) \| H \phi^{\mathrm T} + \phi H^{\mathrm T} \|_F^2 \\
\le & \; \frac 1 {\tau} (1-\epsilon_0)^{-\frac 32} e^{\tau} (e^{2\tau}-1)  m \| H \|_F^2.
\end{align}
Since $\mathrm{Tr}(HH^{\mathrm T} ) = \| H\|_F^2$, it follows that
\begin{align}
 h^{\prime\prime}(s) \le 
\frac 1 {\tau} \Bigl( 1+ (1-\epsilon_0)^{-\frac 32} e^{\tau} (e^{2\tau}-1) m \Bigr) \| H \|_F^2.
\end{align}
The inequality \eqref{t3} follows from a simple Taylor expansion of $h(s)$, namely
\begin{align}
h(1)  \le  h(0) + h^{\prime}(0)  + \frac 12 \max_{0\le s \le 1} h^{\prime\prime}(s).
\end{align}

\end{proof}

\begin{thm}[Modified energy dissipation for matrix-valued AC with mild splitting step constraint]
Suppose  $\Omega =[-\pi, \pi]^d$ is the $2\pi$-periodic 
$d$-dimensional torus in physical dimensions $d\le 3$. 
Let  $U^0: \; \Omega \to  \mathbb R^{m\times m}$ satisfy $ \| \|U^{0}(x)\|_F \|_{L^{\infty}_x} \le \sqrt m$. 
Recall $\mathcal S_{\mathcal L}(t) = e^{t \Delta}$ and  for $A\in
\mathbb R^{m\times m}$, 
\begin{align}
\mathcal S_{\mathcal N}(t) A\coloneqq \left((e^{2t}-1) A A^{\mathrm T} +\mathrm  I\right)^{-\frac12} e^t A.
\end{align}
Define for $n\ge 0$ the Strang-splitting iterates
\begin{align}
U^{n+1}  = \mathcal S_{\mathcal L}\left( \tau/2\right) \mathcal S_{\mathcal N}\left( \tau \right)\mathcal S_{\mathcal L}\left( \tau/2\right) U^n.
\end{align}
If $\tau>0$ satisfies $m e^\tau (e^{2\tau}-1)\le 0.43$,  then 
\begin{align}
\widetilde E(\tilde {U}^{n+1})\le \widetilde E(\tilde {U}^{n}),\quad \forall\, n\ge 0,
\end{align}
where 
\begin{align}
& \tilde{U}^n = \mathcal S_{\mathcal L}\left( \tau/2\right)U^n ; \\
&\widetilde E({U}) =\int_\Omega \frac{1}{2\tau}\left\langle (e^{-\tau \Delta} -1)  {U},  {U} \right\rangle_F + \left\langle G({U}),I\right\rangle_F \, dx;\\
& G(U)  = \frac1{2\tau} U U^{\mathrm T} - \frac{e^\tau}{\tau(e^{2\tau}-1)} \left(\left( \mathrm I+(e^{2\tau}-1)U U^{\mathrm T}\right)^{\frac12}-\mathrm  I\right).
\end{align}
In the above $\langle A,B \rangle_F = \mathrm{Tr}(A^{\mathrm T}B)
= \sum_{i,j} A_{ij} B_{ij} $ denotes the usual Frobenius inner product.
\end{thm} 
\begin{proof}
Observe that 
\begin{align}
e^{-\tau \Delta}\tilde{U}^{n+1}  = \left((e^{2\tau}-1) \tilde U^n {(\tilde U^n)}^{\mathrm T} + \mathrm I\right)^{-\frac12} e^\tau \tilde U^n.
\end{align}
We rewrite the above as
\begin{align} \label{t7}
\frac 1 {\tau} (e^{-\tau \Delta}-1)\tilde{U}^{n+1} + \frac 1 {\tau}(\tilde{U}^{n+1}-\tilde{U}^{n}) = \frac 1 {\tau} \left(  {\left((e^{2\tau}-1) \tilde{U}^n {(\tilde U^n)}^{\mathrm T} +\mathrm I\right)^{-\frac12}} {e^\tau\tilde{U}^n}-\tilde{U}^{n} \right).
\end{align}
Taking the Frobenius inner product with  $ 
\tilde U^{n+1} - \tilde U^{n} $ on both sides of \eqref{t7}, we obtain
\begin{align}
&\frac 1 {\tau} \langle  (e^{-\tau \Delta}-1)\tilde{U}^{n+1},\;
\tilde{U}^{n+1}-\tilde{U}^{n} \rangle_F
+ \frac 1 {\tau} \| \tilde{U}^{n+1}-\tilde{U}^{n} \|_F^2 \notag \\
=&\frac 1 {\tau} \left\langle {\left((e^{2\tau}-1) \tilde{U}^n {(\tilde U^n)}^{\mathrm T} + \mathrm I\right)^{-\frac12}} {e^\tau\tilde{U}^n}-\tilde{U}^{n} , \;\;
\tilde{U}^{n+1}-\tilde{U}^{n}  \right\rangle_F.
\end{align}
It is not difficult to check that
\begin{align}
 & \int_{\Omega} \langle  (e^{-\tau \Delta}-1)\tilde{U}^{n+1},\;
\tilde{U}^{n+1}-\tilde{U}^{n} \rangle_F \,dx \\
=&\; \frac 12 \int_{\Omega} 
\langle  (e^{-\tau \Delta}-1)\tilde{U}^{n+1},\;
\tilde{U}^{n+1} \rangle_F\, dx -\frac 12 \int_{\Omega} 
\langle  (e^{-\tau \Delta}-1)\tilde{U}^{n},\;
\tilde{U}^{n} \rangle_F \,dx \\
& \quad +\frac 12 \int_{\Omega} 
\langle  (e^{-\tau \Delta}-1)(\tilde{U}^{n+1}-\tilde{U}^n),\;
\tilde{U}^{n+1} -\tilde{U}^n\rangle_F\, dx \\
\ge &\; \frac 12 \int_{\Omega} 
\langle  (e^{-\tau \Delta}-1)\tilde{U}^{n+1},\;
\tilde{U}^{n+1} \rangle_F \,dx -\frac 12 \int_{\Omega} 
\langle  (e^{-\tau \Delta}-1)\tilde{U}^{n},\;
\tilde{U}^{n} \rangle_F \,dx.
\end{align}
By Lemma \ref{lemt0} and taking $\epsilon_0=0.43$ therein, we have
\begin{align}
&\int_{\Omega} \frac 1 {\tau} \left\langle {\left((e^{2\tau}-1) \tilde{U}^n {(\tilde U^n)}^{\mathrm T} + \mathrm I\right)^{-\frac12}} {e^\tau\tilde{U}^n}-\tilde{U}^{n} , \;\;
\tilde{U}^{n+1}-\tilde{U}^{n}  \right\rangle_F  \,dx\\
\le & \; \int_{\Omega} G(\tilde U^{n})  \,dx  - \int_{\Omega} G(\tilde U^{n+1}) \,dx
+\frac 1 {2\tau} \Bigl( 1+ (1-\epsilon_0)^{-\frac 32} \epsilon_0 \Bigr) \int_{\Omega} \| \tilde U^{n+1}
-\tilde U^n\|_F^2 \,dx \\
\le & \; \int_{\Omega} G(\tilde U^{n})  \,dx  - \int_{\Omega} G(\tilde U^{n+1}) \,dx
+\frac 1 {\tau}  \int_{\Omega} \| \tilde U^{n+1}
-\tilde U^n\|_F^2 \,dx.
\end{align}
It follows that
\begin{align}
\widetilde E(\tilde {U}^{n+1})\le \widetilde E(\tilde {U}^{n}).
\end{align}
\end{proof}

\begin{rem}
To put things into perspective, we explain the connection of the current work to the companion work \cite{LQ1d}. 
In the scalar case \cite{LQ1d}, we considered the scalar Allen-Cahn equation with both the polynomial potential and the logarithm potential. 
The contributions therein include not only the energy stability, but also the maximum principle for the  logarithm potential where a novel diagonal implicit Runge-Kutta method is proposed.

Concerning the general tensorial models, the current manuscript is inspired from the recent work of Osting and Wang \cite{OW20}. 
On the other hand, the proof of energy dissipation for matrix-valued case is highly nontrivial due to the non-commutativity of general matrices.   
For this, we developed a new machinery and several new monotonicity formulae to establish coercive $H^1$ control on the solution along with maximum principle estimates. 
\end{rem}

%

\section{Numerical experiments}

\subsection{Vector-valued AC equation}\label{sect4.1}
Consider the vector-valued AC equation 
\begin{align}\label{4.1}
\begin{cases}
\partial_t \mathbf u = \Delta \mathbf u + \mathbf u - |\mathbf u|^2 \mathbf u,
\qquad (t, x ) \in (0,\infty) \times \Omega;  \\
\mathbf u \Bigr|_{t=0} =\mathbf u^0, \quad x \in \Omega.
\end{cases}
\end{align}
on the 1-periodic torus $\Omega = \left[-\pi,\pi\right]^2$. 
We use the Strang splitting method given to this equation with a fixed
splitting time step $\tau = 10^{-4}$. 
For the spatial discretization, we use the pseudo-spectral method with $256\times 256$ Fourier modes. 
We take a uniformly distributed random vector $\mathbf v^0$  defined at each grid point.
The initial condition is given by
\begin{align}
\mathbf u^0 = 
\begin{cases} 
0.8 \frac {\mathbf v^0} {|\mathbf v^0|}, \quad \text{if $|\mathbf v^0| \ne 0$};\\
\mathbf 0, \quad \text{otherwise}.
\end{cases}
\end{align}
In this way $\mathbf u^0$ has a fixed magnitude $0.8$ with randomly distributed directions.

Figure \ref{fig:sol_ACvector} shows the computed vector field $\mathbf u$ at $t=0,~0.004,~0.008,~0.016,~0.032,$ and $0.05$ respectively. Define the standard energy 
and the modified energy:
\begin{align}
E(\mathbf u) &= \int_{\Omega} \left( \frac 12 |\nabla \mathbf u|^2 + \frac 14
( | \mathbf u |^2 -1)^2 \right) \,dx; \\
 \widetilde E (\mathbf u) &=\int_\Omega \left( \frac{1}{2\tau}\left\langle (e^{-\tau \Delta} -1)  {\mathbf u},  {\mathbf u} \right\rangle + G({\mathbf u})  +\frac 14 \right)\, dx;\\
&  =\int_{\Omega} \left( \frac{1}{2\tau}\left\langle (e^{-\tau \Delta} -1)  {\mathbf u},  {\mathbf u} \right\rangle
+ \frac1{2\tau} |\mathbf u|^2 - \frac{e^\tau}{\tau(e^{2\tau}-1)} \left(\left(
1+(e^{2\tau}-1)|\mathbf u|^2\right)^{\frac12}-1\right)  +\frac 14\right)\,dx.
\end{align}
It should be noted that a harmless constant $1/4$ is added in the definition of 
$\widetilde E$ to ensure the consistency with the standard energy.
It can be observed that the initial disordered state becomes ordered quickly.
Figure \ref{fig:energy_ACvector} plots the evolution of the standard and modified energies as well as their difference $\Delta E = |\widetilde E-E|$.
Reassuringly  both energy functionals decrease monotonically in time. 

\begin{figure}[!h]
\includegraphics[trim={1in 0.8in 1in 0},clip,width=\textwidth]{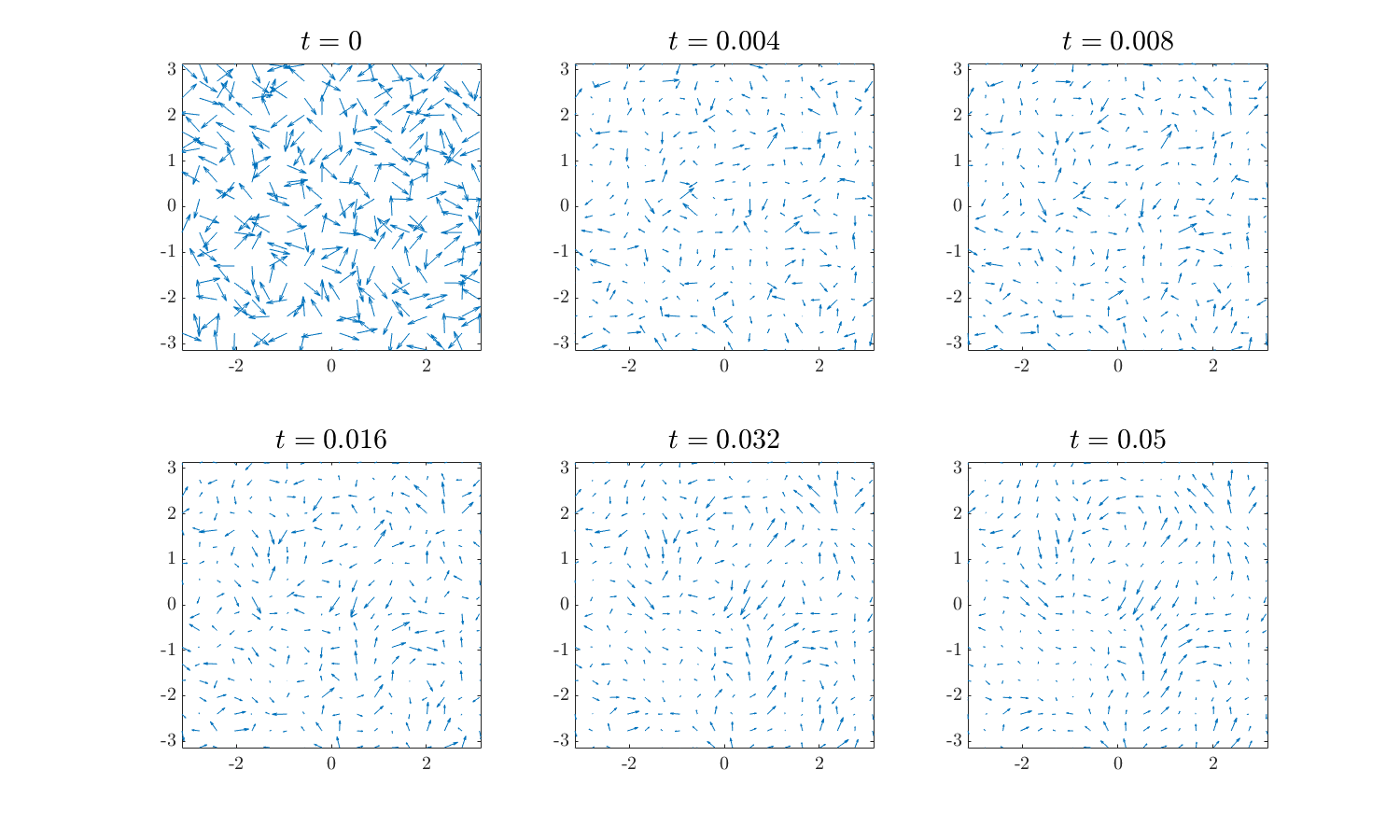}
\caption{Vector field $\mathbf u$ at $t=0,~0.004,~0.008,~0.016,~0.032,$ and $0.05$ respectively for the vector-valued AC equation. }\label{fig:sol_ACvector}
\includegraphics[trim={0in 0 0in 0},clip,width=0.48\textwidth]{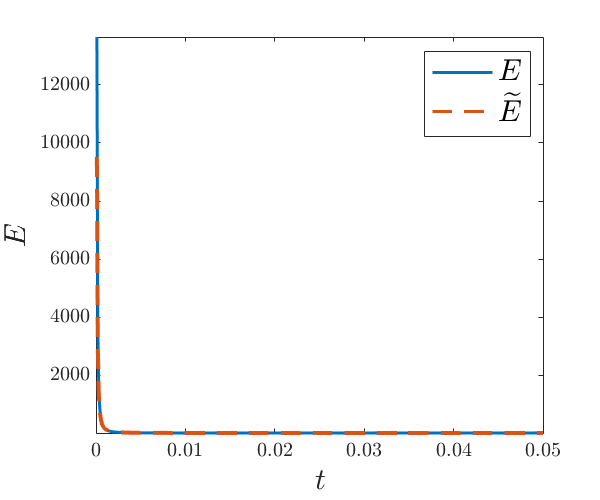}
\includegraphics[trim={0in 0 0in 0},clip,width=0.48\textwidth]{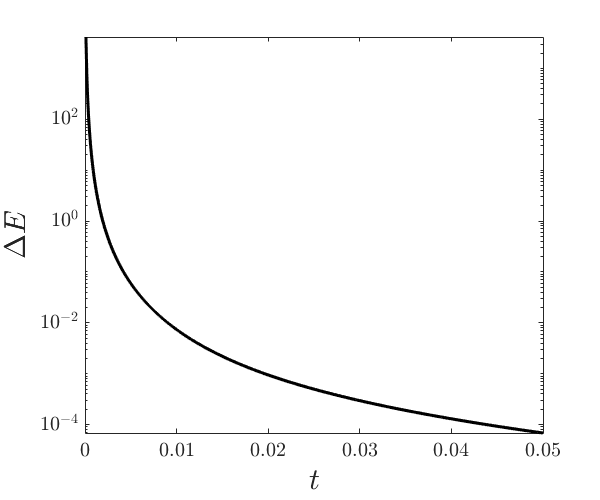}
\caption{Evolution of the original and modified energy as well as their difference $\Delta E = |\widetilde E-E|$ for the vector-valued AC equation. }\label{fig:energy_ACvector}
\end{figure}

We now test the convergence order of the Strang splitting method for vector-valued Allen-Cahn equation with the same settings as above. 
Since the exact PDE solution is not available, we take a small splitting step $\tau = 10^{-6}$ to obtain an ``almost exact'' solution at $t = 0.01$. 
Then, we take several different splitting steps $\tau = \frac{1}{100}\times2^{-k}$ with $k = 5,6,\ldots,10$ to obtain corresponding numerical solutions at $t = 0.01$. 
The $\ell_2$-errors between these solutions and the ``almost exact'' solution are summarized in Table \ref{tab2}. 
It can be observed that the convergence rate is about $2$.

\begin{table}[htb!]
\renewcommand\arraystretch{1.5}
\begin{center}
\def\temptablewidth{1\textwidth}
\caption{$\ell_2$-errors of numerical solutions at time $t = 0.01$ to the vector-valued AC equation \eqref{4.1} for different splitting steps computed by the Strang splitting method.}\label{tab2}
\vspace{-0.2in}
{\rule{\temptablewidth}{1.1pt}}
\begin{tabular*}{\temptablewidth}{@{\extracolsep{\fill}}ccccccc}
   $\tau$     &$\frac1{3200}$     &$\frac1{6400}$
   &$\frac1{12800}$  & $\frac1{25600}$    &$\frac1{51200}$    &$\frac1{102400}$\\  \hline
  $\ell_2$-error   & $3.298\times 10^{-6}$   & $1.199\times 10^{-6}$   & $3.384\times 10^{-7}$ & $8.932\times 10^{-8}$ & $2.277\times 10^{-8}$ & $5.684\times10^{-9}$\\[3pt]
rate  & -- &$1.4604$ &$1.8245$ &$1.9218$ & $1.9716$ & $2.0024$ 
\end{tabular*}
{\rule{\temptablewidth}{1pt}}
\end{center}
\end{table}

\subsection{Matrix-valued AC equation}\label{sect4.2}
Consider the matrix-valued AC equation
\begin{align}
\begin{cases}
\partial_t U = \Delta U  + U - UU^{\mathrm T} U, \quad (t, \mathbf x) \in (0,\infty) \times \Omega;\\
U \Bigr|_{t=0}= U^0.
\end{cases}
\end{align}
The spatial domain $\Omega=[-\pi, \, \pi]^2$ is the $2\pi$-periodic torus in dimension two.
By a slight abuse of notation, we set the initial condition in polar coordinates as
\begin{equation}\label{eq:init1}
U^0(r,\theta) = \left\{
\begin{aligned}
& \left[\begin{array}{cc}\cos\alpha & -\sin\alpha \\ \sin\alpha & \cos\alpha\end{array}\right]  && \mbox{if } r<0.6\pi+0.12\pi\sin(6\theta);\\
& \left[\begin{array}{cc}\cos\alpha & \sin\alpha \\ \sin\alpha & -\cos\alpha\end{array}\right]  
&& \mbox{otherwise,}
\end{aligned}
\right.
\end{equation}
Here $\alpha(x,y) = \frac\pi 2\sin(x+y)$, and $(r,\theta)$ is the polar coordinate of $\mathbf x=(x,y)$. 
For spatial discretization we use the pseudo-spectral method with $256\times256$ Fourier modes.  
The splitting time step is fixed as $\tau = 0.01$.
In Figure \ref{fig:sol_ACmatrix}, the domain is colored by the sign of the determinant of $U$, that is, 
\begin{equation}
\begin{aligned}
& \mbox{yellow} && \mbox{if } \det(U(t,x,y))>0; \\
& \mbox{blue} && \mbox{if } \det(U(t, x,y))<0.
\end{aligned}
\end{equation}
The vector field is generated by the first column vector of the matrix $U(t,x,y)$.  Note that
for $t=0$ this is just 
\begin{equation}
\left(\begin{array}{c} \cos\alpha \\ \sin\alpha \end{array}\right).
\end{equation}
It can be observed that the initial star-shaped line defect shrinks in time.
The evolution of the standard and the modified energy as well as their difference $\Delta E = |\widetilde E-E|$ are plotted in Figure \ref{fig:energy_ACmatrix}.  Clearly these two
energy functionals are in good agreement for small $\tau>0$. 

\begin{figure}[!]
\includegraphics[trim={1in 0 1in 0},clip,width=\textwidth]{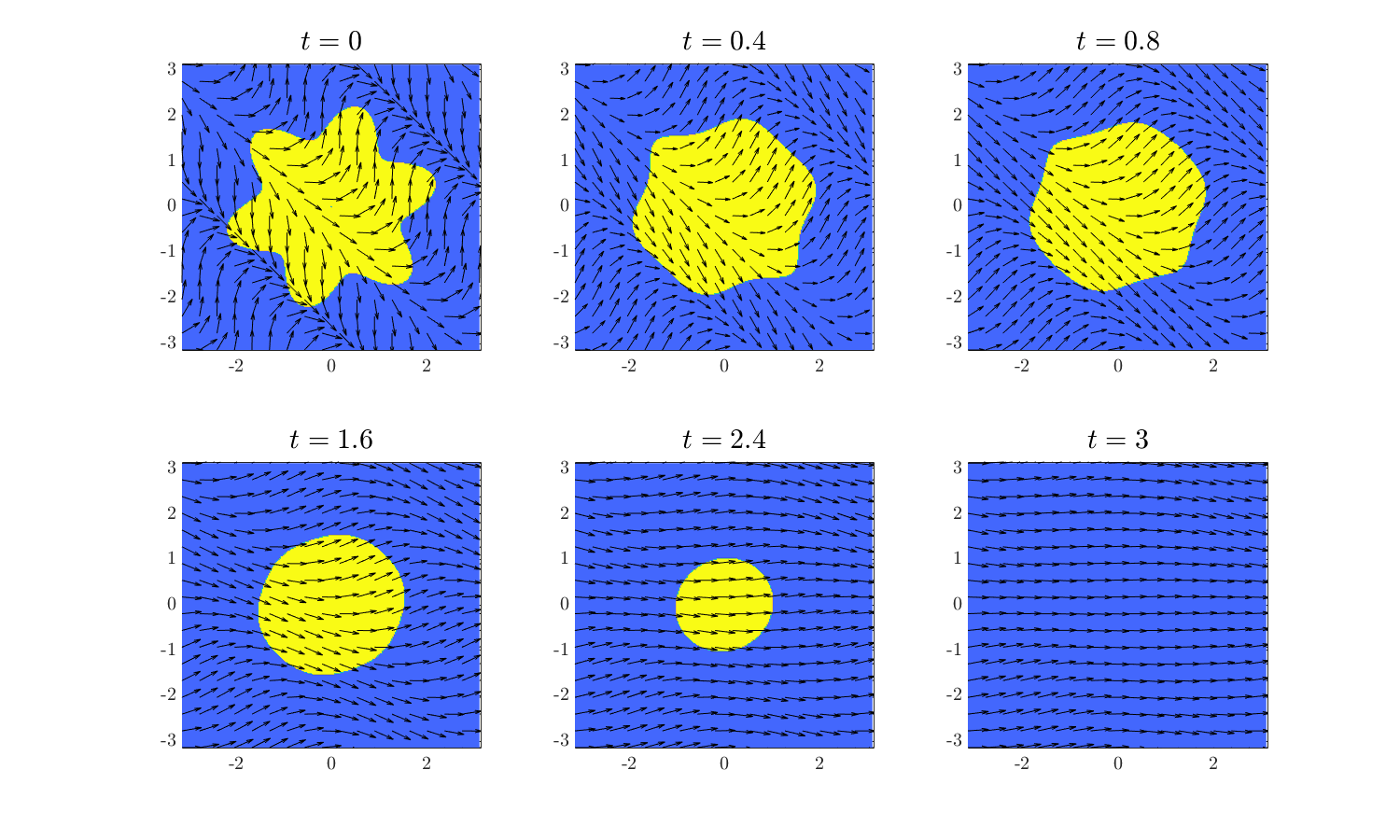}
\caption{Dynamics of line defect for the matrix-valued AC equation  at $t=0,~0.4,~0.8,~1.6,~2.4,$ and $3$ respectively with initial condition \eqref{eq:init1} in Section \ref{sect4.2}. }\label{fig:sol_ACmatrix}
\includegraphics[trim={0in 0 0in 0},clip,width=0.48\textwidth]{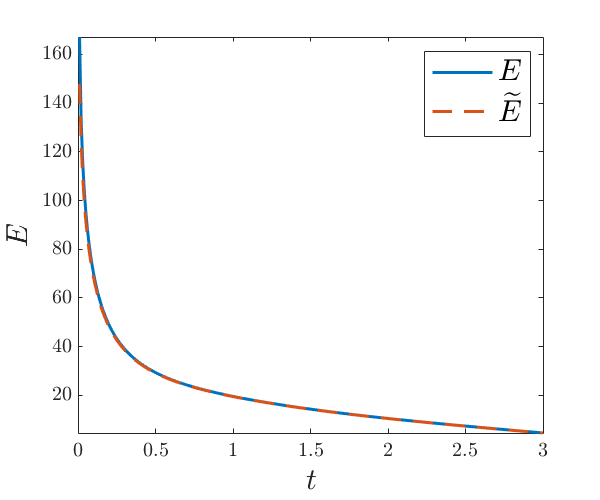}
\includegraphics[trim={0in 0 0in 0},clip,width=0.48\textwidth]{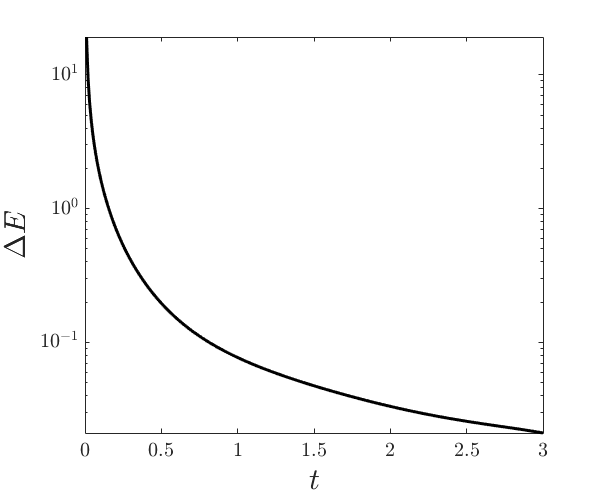}
\caption{Evolution of the original and modified energy as well as their difference $\Delta E = |\widetilde E-E|$ for the matrix-valued AC equation with initial condition \eqref{eq:init1}  in Section \ref{sect4.2}. }\label{fig:energy_ACmatrix}
\end{figure}

Next, we consider the initial condition given by the following. 
\begin{equation}\label{eq:init2}
U^0(r,\theta) = \left\{
\begin{aligned}
& \left[\begin{array}{cc}\cos\alpha & -\sin\alpha \\ \sin\alpha & \cos\alpha\end{array}\right]  && \mbox{if } |x|>0.5\pi |\sin(1.25y)|+0.4\pi,\\
& \left[\begin{array}{cc}\cos\alpha & \sin\alpha \\ \sin\alpha & -\cos\alpha\end{array}\right]  
&& \mbox{otherwise,}
\end{aligned}
\right.
\end{equation}
where $\alpha = y$.
The splitting time step is $\tau = 0.01$ and we take $256\times 256$ Fourier modes.
The dynamics of the line defect and the evolution of the energy are illustrated  in Figure \ref{fig:sol_ACmatrix2} and \ref{fig:energy_ACmatrix2} respectively.
It can be observed that the modified energy dissipation indeed holds in this case. 

\begin{figure}[!]
\includegraphics[trim={1in 0 1in 0},clip,width=\textwidth]{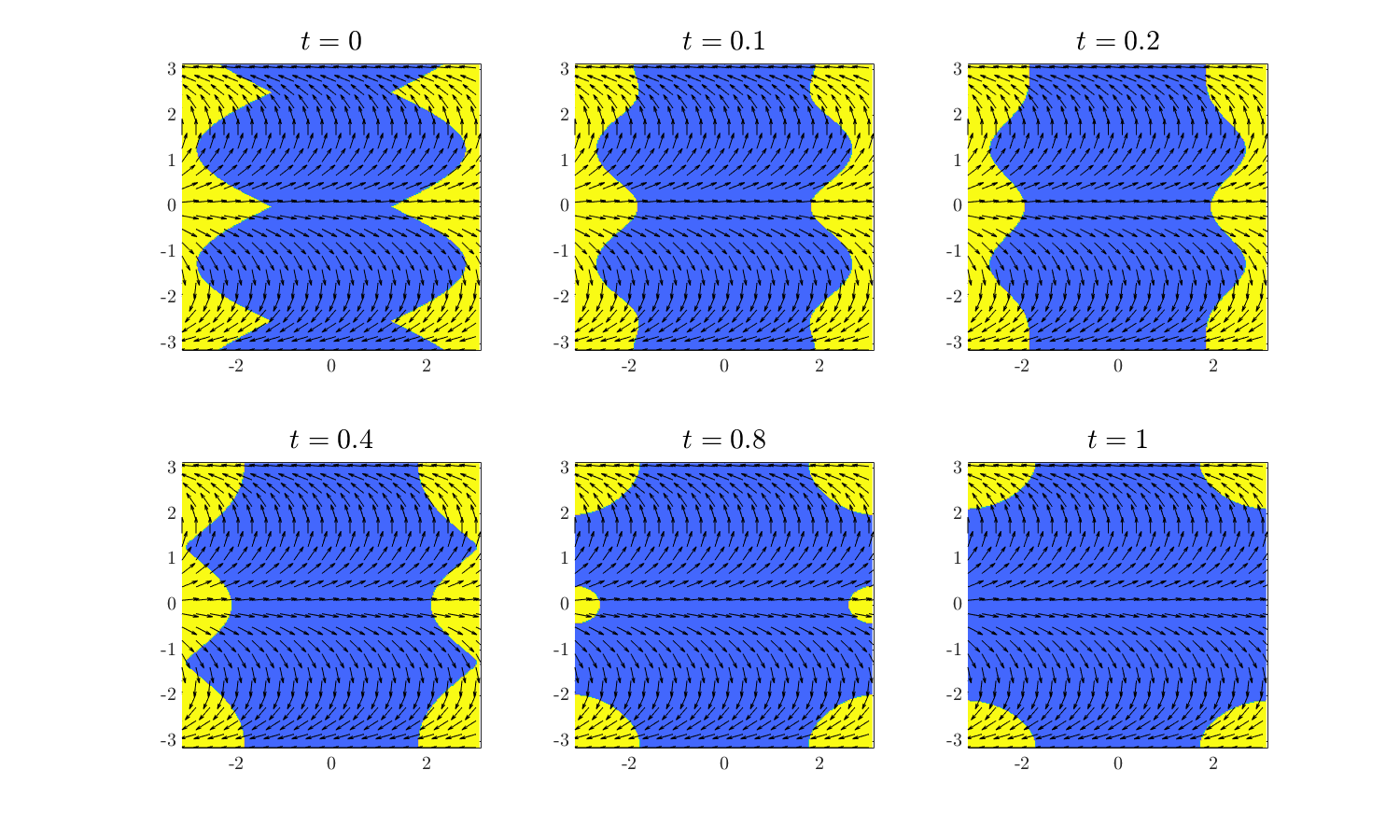}
\caption{Dynamics of line defect for the matrix-valued AC equation  at $t=0,~0.1,~0.2,~0.4,~0.8,$ and $1$ respectively with initial condition \eqref{eq:init2} in Section \ref{sect4.2}. }\label{fig:sol_ACmatrix2}
\includegraphics[trim={0in 0 0in 0},clip,width=0.48\textwidth]{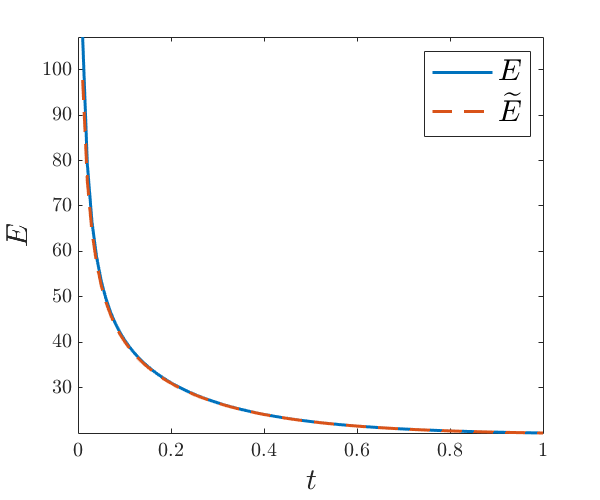}
\includegraphics[trim={0in 0 0in 0},clip,width=0.48\textwidth]{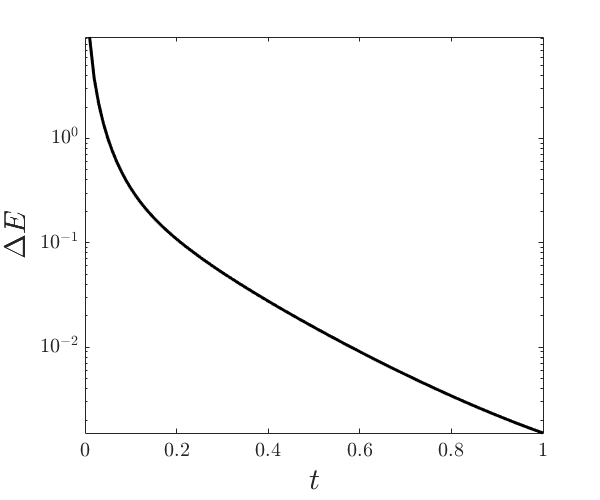}
\caption{Evolution of the original and modified energy as well as their difference $\Delta E = |\widetilde E-E|$ for the matrix-valued AC equation  with initial condition \eqref{eq:init2}  in Section \ref{sect4.2}. }\label{fig:energy_ACmatrix2}
\end{figure}

\section*{Acknowledgements}
The research of C. Quan is supported by NSFC Grant 11901281, the Guangdong Basic and Applied Basic Research Foundation (2020A1515010336), and the Stable Support Plan Program of Shenzhen Natural Science Fund (Program Contract No. 20200925160747003).


\end{document}